\documentclass[10pt,oneside,leqno]{amsart}
\usepackage{amsxtra}
\usepackage{amsopn}
\usepackage{amsmath,amsthm,amssymb}
\usepackage{amscd}
\usepackage{amsfonts}
\usepackage{latexsym}
\usepackage{verbatim}

\theoremstyle{plain}
\newtheorem{theorem}{Theorem}[section]
\newtheorem{definition}[theorem]{Definition}
\newtheorem{lemma}[theorem]{Lemma}
\newtheorem{prop}[theorem]{Proposition}

\newtheorem{rem}[theorem]{Remark}
\newtheorem{ex}[theorem]{Example}

\newcommand\C{{\mathbb C}}

\newcommand\Q{{\mathbb Q}}
\newcommand\R{{\mathbb R}}
\newcommand\Z{{\mathbb Z}}
\renewcommand\H{{\mathbb H}}
\newcommand\T{{\mathbb T}} 
\newcommand\trace{{\rm tr}}
\newcommand{\alt}{\raise1pt\hbox{$\bigwedge$}}

\newcommand{\w}[1]{\om^{#1}}

\newcommand{\sq}{\kern.5pt\square}

\newcommand{\adj}{\mathop{\mathrm{adj}}}

\renewcommand{\setminus}{\smallsetminus}
\newcommand{\tr}{\mathop{\mathrm{tr}}}

\def\1{{\ol1}}\def\2{{\ol2}}\def\3{{\ol3}}
\newcommand{\ba}{\begin{array}}
\newcommand{\ea}{\end{array}}
\newcommand{\be}[1]{\begin{equation}\label{#1}}
\newcommand{\ee}{\end{equation}}

\newcommand{\ol}{\overline}

\newcommand{\we}{\wedge}
\newcommand{\bs}{\backslash}

\newcommand{\om}{\omega}

\newcommand{\Ga}{\Gamma}

\newcommand{\La}{\Lambda}

\newcommand{\X}{\mathbf{X}}

\begin{document}
\title[On astheno-K\"ahler metrics]{On astheno-K\"ahler metrics}
\author{Anna Fino and Adriano Tomassini}
\date{\today}
\address{Dipartimento di Matematica \\ Universit\`a di Torino\\
Via Carlo Alberto 10\\
10123 Torino\\ Italy} \email{annamaria.fino@unito.it}
\address{Dipartimento di Matematica\\ Universit\`a di Parma\\ Parco Area delle Scienze 53/A\\
43124 Parma\\ Italy} \email{adriano.tomassini@unipr.it}
\subjclass[2000]{53C55,
53C25, 32C10}
\keywords{astheno-K\" ahler, strong K\"ahler with torsion, deformation, nilmanifold}
\thanks{This work was supported by the Projects MIUR ``Riemannian Metrics and Differentiable Manifolds'',
``Geometric Properties of Real and Complex Manifolds'' and by GNSAGA
of INdAM}
\begin{abstract} A Hermitian metric on a complex manifold
of complex dimension $n$
is called {\em astheno-K\"ahler}
if its fundamental $2$-form $F$ satisfies the condition $\partial
\overline \partial F^{n - 2} =0$. If $n =3$, then the metric is {\em strong KT}, i.e. $F$ is $\partial
\overline \partial$-closed. By using
blow-ups and the twist construction, we construct simply-connected
astheno-K\"ahler manifolds of complex dimension $n > 3$. Moreover, we construct a family of
astheno-K\"ahler (non strong KT) $2$-step nilmanifolds of complex dimension $4$ and we study deformations of strong KT structures on nilmanifolds of complex dimension $3$.

Finally, we study the relation between astheno-K\"ahler
condition and (locally) conformally balanced one and we provide
examples of locally conformally balanced
astheno-K\"ahler metrics on $\T^2$-bundles over (non-K\"ahler) homogeneous complex surfaces. \end{abstract}
\maketitle
\section{Introduction}
Let $(M,J,g)$ be a Hermitian manifold of complex dimension $n$. By \cite{Ga2} there exists
a one-parameter family of canonical Hermitian connections
$$\nabla^t = t \nabla^C + (1 - t) \nabla^0,
$$ where $\nabla^C$ and $\nabla^0$ denote the {\em Chern
connection} and the {\em first canonical connection} respectively. This family includes for $t = -1$ the
{\em Bismut connection} $\nabla^B $ considered by  J.M. Bismut in \cite{Bi}.

If the fundamental $2$-form $F (\cdot, \cdot) = g(J \cdot, \cdot)$ is closed, then the metric
$g$ is K\"ahler and any connection $\nabla^t$ in the above family coincides with the
Levi-Civita connection. In the literature, weaker conditions on $F$ have been studied and they involve
the closure with respect to the $\partial \overline \partial$-operator of the
$(k , k)$-form $F^k = F \wedge \cdots \wedge F$. Some of these conditions are characterized by some
properties of either the Chern or Bismut connection.

More precisely, if $\partial \overline \partial F =0$, then the
Hermitian structure $(J, g)$ is said to be {\em strong K\"ahler with torsion} and
$g$ is called {\em strong KT} (see e.g. \cite{GHR}). In this case the Hermitian structure is characterized
by the condition that the Bismut connection has skew-symmetric
torsion. Strong KT metrics have been recently studied by many
authors and they have also applications in type II string theory
and in 2-dimensional supersymmetric $\sigma$-models
\cite{GHR,Str, IP}. Moreover, they have also links with generalized
K\"ahler structures (see for instance
\cite{GHR,Gu,Hi2,AG,FPS,FT}). New simply-connected strong KT examples have been recently
constructed by A. Swann in \cite{Sw} via the twist construction, by reproducing the $6$-dimensional examples
found previously in \cite{GGP}.

If $\partial \overline \partial F^{n - 2} =0$, then in the
terminology of J. Jost and S.-T. Yau (\cite{JY,LYZ}) the Hermitian
metric $g$ on $M$ is said to be {\em astheno-K\"ahler}. Therefore,
on a complex surface any Hermitian metric is automatically astheno-K\"ahler and in
complex dimension $n =3$ the notion of astheno-K\"ahler metric
coincides with that one of strong KT.
For $n > 3$, as far as we
know, not many results and examples of astheno-K\"ahler manifolds are
known.

Some rigidity theorems concerning compact
astheno-K\"ahler manifolds have been showed in \cite[Theorem
6]{JY} and in \cite{LYZ}, where, in particular, a generalization
to higher dimension of the Bogomolov's Theorem on $VII_0$ surfaces
is proved (see \cite[Corollary 3]{LYZ}). Astheno-K\"ahler structures on Calabi-Eckmann manifolds have been constructed in \cite{Matsuo}.

In \cite{FT2} we proved that the
blow-up of a complex manifold $M$ at
a point or along a compact complex submanifold $Y$ is still strong KT, as in the K\"ahler case
(see for example \cite{Bl}).
In Section \ref{blowup} we will show that the results of \cite{FT2} about resolutions of strong KT
orbifolds can be extended to Hermitian orbifolds satisfying the conditions
\begin{equation}\label{specialastheno}
\partial\overline{\partial}\,F=0\,,\quad \partial\overline{\partial}\,F^2=0.
\end{equation}
We will show that these manifolds satisfy  $\partial \overline \partial F^k = 0$ for all 
$k > 1$ and therefore they  are a proper subset of the astheno-K\"ahler 
manifolds.

As an application, we will construct a simply-connected
example of compact astheno-K\"ahler manifold satisfying previous
conditions. Moreover, we will show that other $8$-dimensional examples may be obtained by
applying the twist construction of \cite{Sw} to astheno-K\"ahler manifolds with torus action.

In complex dimension $3$ invariant astheno-K\"ahler structures on {\em nilmanifolds}, i.e. on compact quotients
of nilpotent Lie groups by uniform discrete subgroups, were studied
in \cite{FPS} showing that the existence of such a
structure depends only on the left-invariant complex structure on the Lie group. In Section \ref{blowup} we will construct a family of
astheno-K\"ahler $2$-step nilmanifolds of complex dimension $4$, showing that in
general, for $n >3$, there is
no relation between the astheno-K\"ahler and strong KT condition (Theorem \ref{asthenofamily})
and that is not anymore true that if $(J, g)$ is astheno-K\"ahler, then any other $J$-Hermitian metric $\tilde g$
is astheno-K\"ahler.

By the classification obtained in \cite{FPS} one of the strong
KT nilmanifolds is the {\em Iwasawa manifold}. In contrast with the
Kodaira-Spencer stability theorem \cite{KodS} and the case of
complex surfaces, in \cite{FT2} we proved that on the Iwasawa
manifold the condition strong KT is not stable under small
deformations of the complex structure. Deformations of complex
structures on nilmanifolds have been studied in \cite{Sal, CF,
MPPS,CFP} and recently S. Rollenske proved in \cite{Rollenske,Rollenske2} that,
in the generic case, small deformations of invariant complex
structures on nilmanifolds are again invariant. This result can
be applied to the strong KT $6$-dimensional nilmanifolds and then
we have that any small deformation and deformation in large of an
invariant strong KT complex structure $J_0$ on a $6$-dimensional
nilmanifold is still invariant. \newline By using the results of
\cite{KS, Ug} we will prove that the space of deformations of a
strong KT complex structure $J_0$ on a $6$-dimensional nilmanifold
for which there exists a strong KT metric is parametrized
generically by a real algebraic hypersurface of degree $4$ in
$\C^4$ through the origin (Theorem \ref{hypersurface}).
Furthermore, we show that the origin is non singular (respectively
singular) according to the fact that $J_0$ is non abelian
(respectively abelian).

If $F^{n -1}$ is
$\partial \overline\partial$-closed or equivalently if its Lee form is co-closed, then the Hermitian metric $g$ is
called {\emph {standard}} or a {\em Gauduchon metric} \cite{Ga}. The Hermitian structure is said
to be {\em balanced} if its Lee form vanishes and {\em conformally balanced} if its Lee form is exact.
Astheno-K\"ahler and
strong KT metrics on compact complex manifolds cannot be balanced for $n > 2$ unless they are
K\"ahler (see \cite{MT,AI}). Moreover, by \cite{IP,P} a conformally balanced strong KT structure
on a compact manifod of complex dimension $n$ whose Bismut connection has (restricted)
holonomy contained in $SU(n)$ is necessarily K\" ahler.\newline
We will
prove a similar result for the astheno-K\"ahler metrics (Theorem \ref{confbalastheno})
and we will show that any non-K\"ahler compact homogeneous complex surface admits a
non-trivial compact $\T^2$-bundle
$M$ carrying an astheno-K\"ahler metric whose Lee form is closed. In
the case of the ${\mathbb T}^2$-bundle over the secondary Kodaira surface
we will obtain a \lq\lq locally conformal solution\rq\rq of the Strominger's
system considered in \cite{Str}.

\medskip \noindent {\em{Acknowledgements}}. We would like to
thank Gueo Grantcharov and Simon Salamon for useful comments and conversations. We are also grateful to
CIRM-FBK and to the Department of Mathematics of Trento for their warm hospitality. We also would like to thank the referee for valuable  
remarks  which  improved the contents of the paper.

\section{Astheno-K\"ahler manifolds} \label{blowup}
Let $(M, J)$ be a complex manifold of complex dimension $n$. Following
Jost and Yau (see \cite{JY}), we recall the following
\begin{definition}
A Hermitian metric
$g$ on $(M, J)$ is said to be {\em astheno-K\"ahler} if its fundamental
$2$-form form $F$ satisfies the condition
$$
\partial\overline{\partial}F^{n-2}=0.
$$
\end{definition}
Thus, by definition, any Hermitian metric on a complex surface is
astheno-K\"ahler and in complex dimension $3$, an
astheno-K\"ahler structure means a strong KT metric.
\begin{rem} {\rm The product of
two strong KT manifolds is still strong KT. This property is not true anymore
for astheno-K\"ahler metrics. Indeed,
for instance  the product metric on the product of the Hopf surface and the Kodaira-Thurston
surface is strong KT but it is not astheno-K\"ahler.}
\end{rem}

If $n > 3$, then the condition $
\partial\overline{\partial}F^{n-2}=0
$ is equivalent to
$$
d (c \wedge F^{n -3}) =0\,,
$$
where $c = - J dF$ is the torsion $3$-form of the Bismut connection.

Note that, in general, if a Hermitian manifold $(M, J, g)$ satisfies the conditions \eqref{specialastheno}
then one has $\partial \overline \partial\, F^k=0$, for any $k
\geq 1$ and in particular $g$ is astheno-K\"ahler, strong KT and
standard.
This follows by
$$
\partial \overline \partial F^k =k\,\partial\left(\overline{\partial}F\wedge F^{k - 1}\right)=
k\left(\partial\overline{\partial}\,F \wedge F - (k - 1) \overline{\partial}\,F\wedge\partial\,F \right) \wedge F^{k -2}, \quad k > 1\,.$$
Indeed, if \eqref{specialastheno} holds, then
$\partial F \wedge\overline{\partial} F =0$ and therefore any $F^k$ is $\partial \overline \partial$-closed.

Following \cite{Ga} we recall that a Hermitian metric
$g$ on  $(M, J)$ is said to be {\em standard} if $F^{n -1}$ is
$\partial \overline\partial$-closed. 
Then, if $n = 4$ a Hermitian metric which is at the same time strong KT
and astheno-K\"ahler metric, it must be also standard.

A necessary condition for the
existence of astheno-K\"ahler metrics on compact complex manifolds
was found in \cite[Lemma 6]{JY}, proving that any holomorphic
$1$-form must be $d$-closed.

We will provide a compact complex $3$-dimensional manifold satisfying the previous condition on
holomorphic $1$-forms, with no astheno-K\"ahler metrics.

We start to note that, by using similar methods to those ones used in \cite[Theorem 2.2]{FG} and
in \cite[Prop. 21]{Ug} in the context of
strong KT geometry, it is possible to show that if $M$ is a compact quotient $M=\Gamma \backslash G$ of a simply-connected Lie
group $G$ by a uniform discrete subgroup $\Gamma$, endowed with an invariant complex structure $J$ and having
no invariant astheno-K\"ahler $J$-Hermitian metrics, then $M$ does not admit any astheno-K\"ahler
$J$-Hermitian metric at all.
\begin{ex} {\rm Let us consider the $6$-dimensional nilpotent real Lie algebra $\mathfrak{g}$ with structure equations
$$
(0,0,0,0,0,e^{12}+e^{34}),
$$
where, with this notation, we mean that the dual space of $\mathfrak g$ is generated by
$\{e^1,\ldots, e^6\}$ satisfying $$
\left\{
\begin{array}{lll}
de^i &=&0\,,\qquad\quad i=1,\ldots ,5,\\[5pt]
de^6 &=&e^{12}+e^{34}\,,
\end{array}
\right.
$$
where $e^{ij}$ stands for $e^i \wedge e^j$.
Let $G$ be the simply-connected Lie group whose Lie algebra is $\mathfrak{g}$ and set
$$
\eta^j = e^{2j-1} + ie^{2j}\,,\qquad j=1,2,3\,.
$$
Then $\{\eta^1,\eta^2,\eta^3\}$ are complex $(1,0)$-forms that define a left-invariant rational
complex structure $J$ on the nilmanifold
$M=\Gamma\backslash G$, where $\Gamma$ is a co-compact discrete subgroup of $G$ such that $J (\Gamma) \subset \Gamma$. \newline
In view of \cite[Theorem 3.2]{FPS}, there are no strong KT metrics on
$(M,J)$. On the other hand, by \cite[Theorem 2]{CF}, it turns out
that the Dolbeault cohomology group $H^{1,0}_{\overline{\partial}}(M)$ is spanned by
$\{\eta^1,\eta^2\}$. Therefore, any holomorphic $1$-form on $M$ is $d$-closed.
}
\end{ex}
\subsection{Examples by blow-ups and resolutions} The proof of the result by \cite{FT2} about
the blow-up of a strong KT manifold at a point or along a compact complex
 submanifold can be adapted to  the class of Hermitian manifolds whose fundamental $2$-form $F$  satisfies the
conditions \eqref{specialastheno}, since in both cases the new fundamental $2$-form on the blow-up
is obtained by adding a $d$-closed form
to a $\partial\overline{\partial}$-closed form. The Hermitian manifolds satisfying  \eqref{specialastheno} are  exactly  equivalent to those which   satisfy $\partial \overline \partial  F^k = 0$ for all $k \geq 1$ and therefore such manifolds are a proper subset of the astheno-K\"ahler manifolds.

Then one can prove the following
\begin{prop}\label{blowuppoint}
Let $(M,J,g)$ be an astheno-K\"ahler manifold of complex dimension $n$ such that its fundamental $2$-form $F$
satisfies
\eqref{specialastheno}. Then both the blow-up $\tilde M_p$ of $M$ at a point
$p\in M$ and the blow-up $\tilde M_Y$ of $M$ along a compact complex submanifold $Y$ admit an astheno-K\"ahler
metric satisfying \eqref{specialastheno} too.
\end{prop}
Thus by Proposition \ref{blowuppoint} it is possible to construct
new examples of astheno-K\"ahler manifolds by blowing-up a given
astheno-K\"ahler manifold $M$ (satisfying \eqref{specialastheno})
at one or more points or along a compact complex submanifold.
\newline Moreover, one may resolve singularities of a complex
orbifold endowed with a special astheno-K\"ahler metric
(satisfying \eqref{specialastheno}). We recall that orbifolds are
a special class of singular manifolds and they have been used by
Joyce in \cite{J} to construct compact manifolds with special
holonomy, in \cite{CFM,FM} to obtain non-formal symplectic compact
manifolds and in \cite{FT2} to construct new examples of strong KT
manifolds.

One may give the following
\begin{definition}
A Hermitian metric $g$ on an $n$-dimensional complex orbifold $(M,J)$ is said to be {\em astheno-K\"ahler}
if the fundamental
$2$-form $F$ of $g$ satisfies
$$
\partial\overline{\partial}\,F^{n-2}=0\,.
$$
An {\em astheno-K\"ahler resolution} of an astheno-K\"ahler orbifold $(M,J,g)$ is the datum of a
smooth complex
manifold $(\tilde{M},\tilde{J})$ endowed with a $\tilde J$-Hermitian astheno-K\"ahler metric $\tilde{g}$  and of a map $\pi :\tilde{M}\to M$, such that
\begin{enumerate}
\item[i)]$\pi :\tilde{M}\setminus E\to M\setminus S$ is a biholomorphism, where $S$ is the singular set of $M$
and $E=\pi^{-1}(S)$ is
the {\em exceptional set};\\
\item[ii)] $\tilde{g} =\pi^*g$ on the complement of a neighborhood of $E$.
\end{enumerate}
\end{definition}
As in \cite{FT2}, we can apply Hironaka
Resolution of Singularities Theorem \cite{Hiro}, for which the singularities can be
resolved by a finite
number of blow-ups and we may use the previous results about blow-ups
to prove the following
\begin{theorem} \label{resolutionastheno}
Let $(M,J)$ be a complex orbifold of complex dimension $n$ endowed with a $J$-Hermitian astheno-K\"ahler
metric $g$ satisfying
\eqref{specialastheno}. Then there exists an astheno-K\"ahler resolution of $(M,J,g)$ satisfying also
\eqref{specialastheno}.
\end{theorem}
We may apply the previous theorem to the complex orbifold, quotient of the standard complex torus by an involution.
Let $\T^{2n}=\R^{2n}/\Z^{2n}$
be the standard torus and denote by $(x_1,\ldots ,x_{2n})$ global
coordinates on $\R^{2n}$. Consider the complex structure $J$ on $\T^{2n}$ defined by
\begin{equation}\label{torusholomorphic}
\left\{
\begin{array}{lll}
\eta^1 & = & dx_1 +i\left(f(x_n, x_{2n}) dx_{n}+dx_{n+1}\right)\,,\\[5pt]
\eta^j & = & dx_j +i\,dx_{n+j}\,, \quad j=2,\ldots,n,
\end{array}
\right.
\end{equation}
where $f= f (x_n, x_{2n})$ is a ${\mathcal C}^\infty$, $\Z^{2n}$-periodic and even function.\newline
Let $\sigma$ be the $J$-holomorphic involution $\sigma :\T^{2n}\to \T^{2n}$ induced
 by
$$
\sigma\left((x_1,\ldots ,x_{2n})\right)=(-x_1,\ldots ,-x_{2n}).
$$
Thus, $(M=\T^{2n}/\langle\sigma\rangle,J)$ is a complex orbifold
with singular set
$$
S=\left\{ x+\Z^{2n}\,\,\,\vert\,\,\, x\in\frac{1}{2}\Z^{2n}\right\},
$$
 which consists of $256$ points for $n = 4$.
Since
$
\sigma^*(\eta^j)=-\eta^j\,,\, j=1,\ldots ,n\,,
$
the natural Hermitian metric and the corresponding fundamental $2$-form on $\T^{2n}$
$$
g =\frac{1}{2}\sum_{j=1}^n \left(\eta^j\otimes\overline{\eta}^j+\overline{\eta}^j\otimes\eta^j\right), \quad
F=\frac{i}{2}\sum_{j=1}^n \eta^j\wedge\overline{\eta}^j
$$
are both $\sigma$-invariant and
by \cite{FT2} $g$ is strong KT. For $n > 3$, a direct computation gives
$
\partial\overline{\partial}\,F^2= -2\overline{\partial}\,F\wedge\partial\,F=0\,,
$ i.e. the metric $g$ is also astheno-K\"ahler.
According to Theorem
\ref{resolutionastheno}, now we may resolve the singularities of $\T^{2n}/\langle\sigma\rangle$
in order to obtain a
simply-connected astheno-K\"ahler manifold $\tilde{M}$. More precisely, for any singular point $p\in S$, we
take the blow-up at $p$. As in \cite{J} we deduce that the astheno-K\"ahler
resolution $\tilde{M}$ of the orbifold $\T^{2n}/\langle\sigma\rangle$ is
simply-connected.
\subsection{Examples by twist construction} We recall that in general, given a manifold $M$ with
a torus action and a principal torus bundle $P$ with connection $\theta$, if the torus
action lifts to $P$ commuting with the principal action, then one may
construct the twist $W$ of the manifold, as the quotient of $P$ by the torus action
(see \cite{Sw}). Moreover, if the lifted torus action preserves the principal connection
$\theta$, then tensors on $M$ can be transferred to tensors on $W$ if their pullbacks to
$P$ coincide on $\mathcal H = {\mbox {Ker}} \, \theta$. A differential form $\alpha$ on
$M$ is ${\mathcal H}$-{\em related} to a differential form $\alpha_W$ on $W$,
$\alpha \sim_{\mathcal H} \alpha_W$, if their pull-backs to $P$ coincide on $\mathcal H$.

By applying the
twist construction of \cite[Prop. 4.5]{Sw} to $8$-dimensional
astheno-K\"ahler manifolds with torus action, one can get new simply-connected
 astheno-K\"ahler examples.\newline
Let $(N^6, J)$ be a $6$-dimensional simply-connected compact
complex manifold with a $J$-Hermitian structure $g$ which is strong
KT and standard. Consider the product $M^8 = N^6 \times \T^2$, where $\T^2$ is a $2$-torus with an invariant K\"ahler structure. Then $M^8$ is astheno-K\"ahler
and strong KT with torsion $c$ supported on $N^6$.\newline
Assume that there are two linearly independent integral closed $(1,1)$-forms $\Omega_i \in \Lambda^{1,1}_{\Z} (N^6)$,
$i = 1,2$, with $[\Omega_i] \in H^2 (N^6, \Z)$.
If
$$
\sum_{i,j=1}^2 \gamma_{ij} \Omega_i \wedge \Omega_j =0
$$
for some positive definite
matrix $(\gamma_{ij}) \in M_2 (\R)$, then by \cite[Prop. 4.5]{Sw} there is a compact simply
connected $\T^2$-bundle $\tilde W$ over $N^6$ whose total space is strong KT.
The manifold $\tilde W$ is the universal covering
of the twist $W$ of $N^6 \times \T^2$, where the K\"ahler flat
metric over $\T^2 = \C/\Z^2$ is given by the matrix $(\gamma_{ij})$
with respect to the standard generators with a compatible complex
structure and
topologically $W$ is a principal torus bundle over $N^6$ with Chern classes $[\Omega_i]$. Under the additional condition
\begin{equation}
\label{twistcond}
c \wedge \Omega_j =0\,,\qquad j = 1,2,
\end{equation}
we will prove that the total space is astheno-K\"ahler.\newline
 By \cite[Prop. 4.2]{Sw}, $W$ has torsion $3$-form $c_W$ such that
$$
\begin{array}{l}
c_W \sim_{\mathcal H} c - a^{-1} \Omega \wedge \xi^{b},\\[5pt]
d c_W \sim_{\mathcal H} dc + \sum_{i,j=1}^2 \gamma_{ij} \Omega_i \wedge \Omega_j,
\end{array}
$$
where $a^{-1} \Omega = (\Omega_1, \Omega_2)$ and $\xi$ is
the standard action of the torus on the $\T^2$-factor.
Denote by $F = F_{N^6} + F_{\T^2}$ and $F_W$ respectively the fundamental $2$-form associated to the
Hermitian structure $(J, g)$ on $M^8$ and the induced Hermitian structure $(J_W, g_W)$ on $W$.
Then, since $F_W \sim_{\mathcal H} F$ we have that also $F_W^2 \sim_{\mathcal H} F^2$. Therefore,
by \cite[Prop. 4.2]{Sw}
$$
c_W \wedge F_W \sim_{\mathcal H} (c - a^{-1} \Omega \wedge \xi^{b}) \wedge F.
$$
By using again \cite[Cor. 3.6]{Sw}, it follows that
$$
d (c_W  \wedge F_W) \sim_{\mathcal H} d ( (c - a^{-1} \Omega \wedge \xi^{b}) \wedge F) -
a^{-1} \Omega \wedge i_{\xi} ( (c - a^{-1} \Omega \wedge \xi^{b}) \wedge F)\,,
$$
where $i_{\xi}$ denotes the contraction by $\xi$. \newline
Now, $i_{\xi} c =0$ and $i_{\xiÊ} F = J \xi^b$ and thus
\begin{equation}\label{astheno-twist}
\begin{array}{lcl}
d(c_W \wedge F_W) &\sim_{\mathcal H} & (dc + \sum_{i,j=1}^2 \gamma_{ij} \Omega_i \wedge \Omega_j)
\wedge F - (c - a^{-1} \Omega \wedge \xi^b) \wedge d F\\[7pt]
&& + a^{-1} \Omega \wedge (c - a^{-1} \Omega \wedge \xi^b)\wedge i_{\xi} F\,.
\end{array}
\end{equation}
Observe that
\begin{equation}\label{Omega}
a^{-1} \Omega \wedge a^{-1} \Omega \wedge \xi^b \wedge i_{\xi} F =
\sum_{i,j=1}^2 \gamma_{ij} \Omega_i \wedge \Omega_j =0\,,
\end{equation}
and that $d F_{N^6} \wedge a^{-1} \Omega =0$, since $a^{-1} \Omega$
is of type $(1,1)$. Therefore, by the assumption \eqref{twistcond}, by \eqref{Omega} and by the
astheno-K\"ahlerianity of the metric $g$ on $M^8 = N^6 \times \T^2$, it follows that the right hand side of
\eqref{astheno-twist} vanishes. Hence $\tilde{W}$ is strong KT and astheno-K\"ahler.

\subsection{$8$-dimensional nilmanifolds}
We will construct a family of
astheno-K\"ahler (non strong KT) $2$-step nilmanifolds of real dimension $8$, showing that in higher
real dimension than $6$ there is in general no relation between
astheno-K\"ahler and strong KT structures.
Let $\{\eta^1,\ldots ,\eta^4\}$ be the set of complex forms of
type $(1,0)$, such that
\begin{equation}\label{cxstructureequations}
\left\{ \begin{array}{lcl}
d \eta^ j &=&0, \, j = 1,2,3, \\[5pt]
d \eta^ 4 &=& a_1 \, \eta^1 \wedge \eta^2 + \, a_2 \eta^1 \wedge
\eta^3 + \, a_3 \eta^1 \wedge \overline \eta^1 + a_4 \, \eta^1
\wedge \overline\eta^2
+ a_5 \, \eta^1 \wedge\overline \eta^3\\[5pt]
&& + a_6 \, \eta^2 \wedge \eta^3
+ a_7 \, \eta^2 \wedge \overline \eta^1
+ a_8 \, \eta^2 \wedge \overline \eta^2 + a_9 \, \eta^2 \wedge \overline \eta^3
+ a_{10} \, \eta^3 \wedge \overline  \eta^1\\[5pt]
&& + a_{11} \, \eta^3 \wedge \overline \eta^2 + a_{12} \, \eta^3 \wedge \overline \eta^3,
\end{array}
\right.
\end{equation}
where $a_j \in \C$, $j = 1, \ldots, 12$.

Then the complex forms $\{\eta^1,\ldots ,\eta^4\}$ span the dual of a $2$-step
nilpotent Lie algebra $\mathfrak{n}$, depending on the complex
parameters $a_1,\ldots ,a_{12}$ and define an integrable almost
complex structure $J$ on $\mathfrak{n}$.
Let $N$ be the simply connected Lie group with Lie
algebra $\mathfrak{n}$. Then, for any $a_1,\ldots ,a_{12} \in\Q[i]$, by the nilpotency of $N$, in view of Malcev's
theorem \cite{Mal}, there exists a
uniform discrete subgroup $\Gamma$ of $N$ such that $M=\Gamma \backslash N$ is
a compact nilmanifold.
\begin{theorem}\label{asthenofamily}
Let $a_1,\ldots ,a_{12}\in\Q[i]$ and $(M=\Gamma \backslash N, J)$ be the corresponding compact complex
nilmanifold of real dimension $8$.
Then the Hermitian metric $$g=\frac{1}{2}\sum_{j=1}^4\eta^j\otimes\overline{\eta}^j+
\overline{\eta}^j\otimes{\eta}^j$$ is astheno-K\"ahler if and only if
\begin{equation} \label{concoeffastheno}
\begin{array}{c}
\vert a_1 \vert^2+ \vert a_2 \vert^2 +\vert a_4 \vert^2 +\vert a_5 \vert^2 +\vert a_6
\vert^2+ \vert a_7\vert^2 + \vert a_9 \vert^2 +\\[5pt]
 \vert a_{10} \vert^2 + \vert a_{11} \vert^2
 = 2 \Re\mathfrak{e}\,( a_3 \overline{a}_ 8 + a_3 \overline{a}_{12} + a_8 \overline{a}_{12} )\,.
\end{array}
\end{equation}

If, in addition $a_8=0$ and $\vert a_4 \vert^2 + \vert a_{11} \vert^2 \neq 0$, then the astheno-K\"ahler
metric $g$ is not strong KT. Moreover, if $ a_8 = 0$,
the astheno-Kahler metric $g$ is strong KT if and only if $a_1 = a_4 = a_6 = a_7 = a_9 = a_{11} =0$.
\end{theorem}
\begin{proof} The fundamental $2$-form of $(J, g)$ is given by
$$
F = \frac{i}{2}\sum_{j=1}^4
\eta^j\wedge\overline{\eta}^j\,.
$$
A straightforward computation yields
$$
\begin{array}{lcl}
\partial\overline{\partial}\,F^2 &=&-\frac{1}{2}\,\partial\overline{\partial}
\left( \sum_{i<j}\eta^{i\overline{i}j\overline{j}} \right)\\[5pt]
&&=\frac{1}{2} (\vert a_1\vert^2 + \vert a_2\vert^2 \!+\! \vert a_4
\vert^2\! + \!\vert a_5\vert^2\! +\! \vert a_6 \vert^2 +\! \vert a_7 \vert^2 +\! \vert a_9 \vert^2 +\! \vert a_{10}
\vert^2 +\! \vert a_{11} \vert^2 \\[5pt]
&&- 2\Re\mathfrak{e}\,(a_3\overline{a}_8 + a_3\overline{a}_{12} + a_8\overline{a}_{12}))\,
\eta^{123\overline{1}\overline{2}\overline{3}}.
\end{array}
$$
where, for instance, $\eta^{i\overline{i} j\overline{j}}$ denotes the wedge product
$\eta^i \wedge \eta^{\overline{i}}\wedge \eta^j \wedge\eta^{\overline{j}}$.\newline
Hence
$
\partial\overline{\partial}\,F^2 =0
$
if and only if \eqref{concoeffastheno} holds.

The last part of the Theorem can be easily showed by a direct computation.
\end{proof}
As an application of the last result, we explicitly construct an astheno-K\"ahler metric which is not strong KT.
Take
$$
a_1=a_2= a_5 = a_6 = a_7 = a_8 = a_9 = a_{10} =0, \, a_3=a_4=a_{11}=a_{12}=4
$$
and set
$
\eta^j=e^{2j-1}+ ie^{2j}\,, j=1,\ldots ,4\,.
$
Then
 $\mathfrak{n}$
has structure equations
\begin{equation}\label{astheno}
\left\{
\begin{array} {l}
d e^i = 0\,,\quad i=1,\ldots ,6\,,\\[5pt]
d e^7 = 4(e^{13}+e^{24}-e^{35}-e^{46})\,,\\[5pt]
d e^8 = 4(e^{23}-e^{14}+e^{45}-e^{36}-2e^{12}-2e^{56})\,.
\end{array}
\right.
\end{equation}
Let $M=\Gamma \backslash G$ be the associated compact nilmanifold.
Then, the Hermitian metric
$
g=\sum_{i=1}^8e^j\otimes e^j
$
is an astheno-K\"ahler metric on $M$, that is not strong KT, according to Theorem \ref{asthenofamily}.

If we take
$$
a_1 = a_2=a_4= a_6 = a_7 = a_8 = a_9 = a_{11} =0, \, a_3=a_5=a_{10}=a_{12}=2
$$
and set
$
\eta^j=e^{2j-1}+ ie^{2j}\,, j=1,\ldots ,4\,,
$
then we get a Hermitian metric satisfying conditions \eqref{specialastheno}.
\begin{rem} {\rm In real dimension six, by \cite{FPS} the existence of a strong KT structure on a nilpotent Lie
algebra depends only on the complex structure on the nilpotent Lie algebra. The same property is not anymore
true for a astheno-K\"ahler structure on a nilpotent Lie algebra of real dimension eight. Indeed, for the
nilpotent Lie algebra defined by \eqref{cxstructureequations} with the coefficients $a_j, j = 1, \ldots, 12,$
satisfying the condition \eqref{concoeffastheno}, the $J$-Hermitian metric given by
$$
\frac{1}{2} [2(\eta^1\otimes\overline{\eta}^1+
\overline{\eta}^1\otimes{\eta}^1)+ 3 (\eta^2\otimes\overline{\eta}^2+
\overline{\eta}^2\otimes{\eta}^2) + 4 (\eta^3\otimes\overline{\eta}^3+
\overline{\eta}^3\otimes{\eta}^3) + 5 ( \eta^4\otimes\overline{\eta}^4+
\overline{\eta}^1\otimes{\eta}^4)]
$$
is not any more astheno-K\"ahler.}
\end{rem}

\section{Deformations of strong KT complex structures on $6$-dimensional nilmanifolds} \label{deformations}

 We will say that a complex structure $J$ on a nilmanifold $\Gamma \backslash G$ is {\em invariant} if it arises from a corresponding left-invariant complex structure on the Lie group $G$. We recall that a complex structure $J$ on a Lie algebra ${\mathfrak g}$ is called {\em abelian}, if and only if $[JX, JY] = [X, Y]$, for any $X, Y \in {\mathfrak g}$ (see \cite{BDM}) and it is {\em bi-invariant} if $J$ commutes with the adjoint representation.

By using the result of \cite{FPS} together with the
results about \lq \lq symmetrization\rq \rq of non-invariant structures obtained in \cite{FG, Ug},
one has the following
\begin{theorem}\label{SKT} Let $M^6=\Ga\bs G$ be a 6-dimensional nilmanifold with
an invariant complex structure $J$. Then there exists a $J$-Hermitian strong KT metric $g$
if and only if $J$ has a basis $(\omega^ 1, \omega^2, \omega^3)$ of $(1,0)$-forms such that
\be{SKT1}
\left\{\ba{l}d\omega^1=0\\d\omega^2=0\\d\omega^3=A\omega^{\overline1 2} +
B \omega^{ \overline 2 2}+C \omega^{1 \overline 1} + D \omega^{1 \overline 2}  +E \omega^{12}
\ea\right.\ee where $A,B,C,E, F$ are complex numbers such that \be{SKT2}
|A|^2+|D|^2+|E|^2+2{\mathfrak {Re}}(\ol BC)=0
\ee
and $\omega^{i \overline j}$ stands for $ \omega^i \wedge \overline \omega^j$.
Moreover, the Lie algebra $\mathfrak g$ of $G$ is isomorphic to one of the following:
$$
\begin{array}{l}
{\mathfrak h}_2 = (0,0,0,0,e^{12}, e^{34})\,,\\[3pt]
{\mathfrak h}_4 = (0,0,0,0,0, e^{12}, e^{14} + e^{23}),\, \\[3pt]
{\mathfrak h}_5 = (0,0, 0,0, e^{13}+e^{42}, e^{14}+e^{23})\,, \\[3pt]
{\mathfrak h}_8 = (0,0,0,0,0,0, e^{12})\,.
\end{array}
$$
\end{theorem}

\begin{rem} {\rm By the previous theorem one has that $J_0$ is  abelian, i.e. the differential
of the $(1,0)$-forms
are only of type $(1,1)$, if and only if $E =0$.}
\end{rem}
In the sequel, we will denote by $J_0$ the strong KT complex structure, i.e.
$J_0$ is the complex structure which gives rise to a strong KT structure, associated to the basis
$(\omega^1, \omega^2, \omega^3)$ satisfying the condition \eqref{SKT2}.

We will use as in \cite{FPS} the notation
$$
{\bf Y}_{\omega} = A\omega^{\overline1 2} +B \omega^{ \overline 2 2}+C \omega^{1 \overline 1} +
D \omega^{1 \overline 2}\,,
$$
where
$$
{\bf Y} = \left ( \begin{array}{cc} A&B\\ C& D
\end{array} \right)
$$
so that
$$
d \omega^3 = {\bf Y}_{\omega} + E \omega^{12}\,.
$$
Moreover, we will denote as in \cite{FPS} by
$$
{\adj }({\bf Y} )= \left ( \begin{array}{cc} D&-B\\ -C&A
\end{array} \right).
$$
The $1$-forms $\omega^j\,, j = 1,2,3$,
associated to the strong KT complex structure $J_0$, are left-invariant on
$G$ and they define a basis $(e^1, \ldots, e^6)$
 of real $1$-forms by setting
 $$
\omega^1 = e^1 + i e^2 \,,\quad \, \omega^2 = e^3 + i e^4\,,\quad \omega^3 = e^5 + i e^6\,.
$$
These $1$-forms are pull-backs of corresponding $1$-forms on $M^6$, which we denote
by the same symbols.

Since for the dual basis $(e_1, \ldots, e_6)$ we have
$$
[e_j, e_k] \subseteq {\mbox {span}} <e_5, e_6>\,,
$$
for any $j, k = 1, \ldots, 4$, the quotient $M^6$ is the total space of a
principal ${\mathbb T}^2$ -bundle over ${\mathbb T}^4$.
The space of invariant $1$-forms annihilating the fibres of $\pi: M \to {\mathbb T}^4$ is
$$
{\mathbb V} = {\mbox {span}} <e^1, e^2, e^3, e^4> \subseteq \ker (d: \mathfrak g^* \to \Lambda^2 \mathfrak g^*),
$$
with equality for the Lie algebras ${\mathfrak h}_j, j = 2,4,5$.

As in \cite{KS} we can prove the following
\begin{lemma} \label{lemmainv} Let $J$ be any invariant complex structure on a strong
KT nilmanifold $M^6 = \Gamma \backslash G$ . Then the projection $\pi$ induces a complex structure $\tilde J$
on ${\mathbb T}^4$ such that $\pi: (M^6, J ) \to ({\mathbb T}^4, \tilde J)$
is holomorphic.
\end{lemma}

\begin{proof} In order to prove the result is sufficient to show that $\mathbb V$ is $J$-invariant.
By Theorem \ref{SKT} the Lie algebra $\mathfrak g$ is isomorphic to $\mathfrak h_2, \mathfrak h_4, \mathfrak h_5$
or $\mathfrak h_8$.\newline
Any complex structure on $\mathfrak h_8$ is abelian, so the center  given by ${\mbox {span}} <e_5, e_6>$ is
preserved by the complex structure and therefore $\mathbb V$ is $J$-invariant.
By \cite[Lemma 11]{Ug} for any invariant (non bi-invariant) complex structure on
the Lie algebras $\mathfrak h_2$, $\mathfrak h_4$ and $\mathfrak h_5$ there is a
basis $(\eta^1, \eta^2, \eta^3)$ of $(1,0)$-forms such that
$$
\begin{array}{l}
d \eta^j =0\,, \,\,\,\, j = 1,2\,,\\[3pt]
d \eta^3 = \rho \eta^{12} + \eta^{1 \overline 1} + G \eta^{1 \overline 2} + H \eta^{2 \overline 2}\,,
\end{array}
$$
with $\rho = 0,1$ and $G, H \in \C$. Then, since the real space associated to the complex space
spanned by $\eta^1, \eta^2$ coincides with the kernel of
$d: \mathfrak g^* \to \Lambda^2 \mathfrak g^*$, we have that also in this case $\mathbb V$ is $J$-invariant.

If $J$ is bi-invariant, then the Lie algebra $\mathfrak g$ has to be isomorphic to $\mathfrak h_5$ and the result
follows by \cite{KS}.
\end{proof}

Let $(M^6 =\Gamma \backslash G, J_0)$ be a $6$-dimensional nilmanifold with $J_0$ an
invariant strong KT complex structure. By Theorem \ref{SKT} we know that
$\mathfrak g$ is $2$-step nilpotent with $\dim {\mathfrak g}^1 \geq 2$ and with center of dimension $1$ or
$2$. Therefore,
we may apply Theorem 4.1  and 4.3 by Rollenske in \cite {Rollenske2} and conclude that
any small and large deformation of $J_0$ is still invariant. Consequently, we may consider
invariant deformations and work on the space of complex structures of ${\mathfrak g}$
$$
{\mathcal C} ({\mathfrak g}) = \{ J \in {\mbox {End}} ({\mathfrak g}) \, \vert J^2 = - 1, \, N_J =0 \},
$$
where by $N_J$ we denote the Nijenhuis tensor.

We denote by
$
{\mathcal C}^+({\mathfrak g})$ the space of complex structures of $\mathfrak g$ inducing the same
orientation on $\mathbb V$ as $J_0$ and by ${\mathcal C}_0^{\cdot} ({\mathfrak g})$ the open subset
of ${\mathcal C}^+ ({\mathfrak g})$ whose elements are such that there exists a basis
$( \eta^1, \eta^2, \eta^3)$ for which $\eta^{123} \wedge \omega^{\overline{123}} \neq 0$.

We can prove the following
\begin{theorem} \label{exforms} Let $(M^6, J_0)$ be a $6$-dimensional strong KT nilmanifold with $J_0$
defined by the $(1,0)$-forms $(\omega_1, \omega_2, \omega_3)$. If $J \in {\mathcal C} _0^{\cdot} ({\mathfrak g})$,
then there exists a basis of $(1,0)$-forms $(\eta^1, \eta^2, \eta^3)$ such that
\be{aaa}
\left\{\ba{l}\eta^1=\w1+a\w\1+b\w\2\\
\eta^2=\w2+c\w\1+f\w\2\\\eta^3=\w3+x\w\1+y\w\2+u\w\3,\ea \right.\ee
with $a, b, c, f, x, y, u \in \C$ satisfying
\begin{equation} \label{integrability}
-(\det {\bf X}) E + \left (\tr ({\bf X} \overline {\bf Y}) - \overline E \right) u + \tr ({\bf X}\,{\adj}({\bf Y}) ) =0,
\end{equation}
where  
\begin{equation}\label{defmatrixX}
{\bf X}= \left ( \begin{array}{cc} a&b\\ c& f
\end{array} \right).\end{equation}

Therefore, the generic complex structure on $M^6$ has a space of $(1,0)$-forms generated by the previous
forms $\eta^1$, $\eta^2$ and $\eta^3$.
\end{theorem}
\begin{proof} In view of Lemma \ref{lemmainv}, we have that $\eta^1$ and $\eta^2$ can be chosen
so that their real and imaginary
components span $\mathbb V$. The condition
$$
\eta^{123} \wedge \omega^{\overline{123}} \neq 0
$$
implies that $\omega^1, \omega^2, \omega^3$ appear with
non-zero coefficients.

The equation \eqref{integrability}
follows from the integrability condition
$d\eta^3\we \eta^{12}=0$
expressing the fact that $d \eta^3$ has no term involving $\eta^{\overline1 \overline 2}$.

By a direct computation one has that
$$
d \eta^3 \wedge \eta^{12} = -(c E b - c B + c u\overline B - f E a + f A +
f u \overline{D} - u \overline{E} - b C + b u \overline{C} + aD
 + a u \overline{A}) \omega^{12 \overline{1} \overline{2}}.
$$
Therefore the complex structure is integrable if and only if the equation \eqref{integrability} is satisfied.
\end{proof}

Consider the space $\La=\langle \eta^1,\eta^2,\eta^3 \rangle$ generated by the modified complex
1-forms \eqref{aaa}. If $\La$ is maximally complex, then it defines an invariant
almost complex structure on $M^6$ that we will denote by $J_{\X,x,y}$, where
$
{\bf X}$ is given by \eqref{defmatrixX}.

By the previous theorem, the almost complex structure $J_{\X,x,y}$ is integrable if and only if the equation
 \eqref{integrability} holds.
\begin{rem}
{\rm In the case of Iwasawa manifold (with the bi-invariant complex structure $J_0$ which is
therefore non strong KT) one has $E = 1$ and ${\bf Y} =0$ and therefore
the integrability condition reduces to the equation $$ u=bc-af=-\det\X $$ already considered in \cite{FPS}.}
\end{rem}
We are ready to prove the following
\begin{theorem} \label{hypersurface} Let $(M^6 = \Gamma \backslash G, J_0)$ be a $6$-dimensional nilmanifold with an invariant strong KT complex structure $J_0$. Then the space of deformations of $J_0$ for which there exists a
strong KT metric is parametrized generically by a real algebraic
hypersurface of degree $4$ in $\C^4$ through the origin
$O=(0,0,0,0)$. Furthermore, $O$ is non singular (respectively
singular) according to the fact that $J_0$ is non abelian
(respectively abelian).
\end{theorem}
\begin{proof} By Theorem \ref{exforms}, the generic
complex structure $J$ on $M^6$ is defined by the $(1,0)$-forms $(\eta^1, \eta^2, \eta^3)$
given by \eqref{aaa}.
We have
$$
\begin{array} {lcl}
\eta^{12} &=& \omega^{12} + a \omega^{\overline 1 2 } +
b \omega^{\overline 22} + c \omega^{1 \overline{1}} + f \omega^{1 \overline 2} +
(af - bc) \omega^{\overline 1 \overline 2}\\[5pt]
&=& \omega^{12} + {\bf X}_{\omega} + (\det {\bf X}) \omega^{\overline 1 \overline 2}.
\end{array}
$$
We write the characteristic polynomial of ${\bf X} {\bf \overline X}$ as $p(x) = x^2 - \gamma x + \delta$, so that
$$
\begin{array}{l}
\gamma= {\mbox {tr}} ({\bf X} {\bf {\overline X}}) = \vert a \vert^2 + \vert f \vert^2 + b \overline c +
\overline b c\,,\\[5pt]
\delta = \det({\bf X} {\bf \overline X}) = \vert a \vert^2 \vert f \vert^2 + \vert b \vert^2 \vert c \vert^2 -
a f \overline {b}
\overline {c} - bc \overline{a} \overline {f}.
\end{array}
$$
The relations
$$
\begin{array}{l}
\eta^{1 \overline 1 2 \overline 2} = (1 - \gamma + \delta) \omega^{1 \overline 1 2 \overline 2}\,,\\[5pt]
\eta^{1 \overline 1 2 \overline 2 3 \overline 3} = (1 - \gamma + \delta) (1 - \vert u \vert^2) \omega^{1
\overline 1 2 \overline 2 3 \overline 3}
\end{array}
$$
express volume changes associated to a switch of basis from $(\omega^i)$ to $(\eta^i)$. As a consequence,
$\Lambda \cap \overline \Lambda = \{0 \}$ if and only if
\begin{equation} \label{volume} \vert u \vert \neq 1 \quad \mbox {and} \quad p(1) \neq 0\end{equation}
and these are the conditions that ensure that the complex structure $J_{{\bf X},x,y}$ is well defined.

From now we suppose that the conditions \eqref{SKT2}, \eqref{integrability} and \eqref{volume} hold.

For simplicity, we will also assume that $x = y = 0$,
and we will denote the complex structure $J_{{\bf X},0,0}$ by $J_{{\bf X}}$, since the strong KT condition is determined
in terms of $d\eta^3$, which does
not involve $x, y$.

As in the proof of Lemma 4.1 of \cite{FPS} consider the two bases
$(\omega^1 , \omega^2 , \overline \omega^1 , \overline \omega^2)$ and
$(\eta^1 , \eta^2 , \overline \eta^1 , \overline \eta^2)$. The second
is related to the first one by the block matrix
$$
{\bf M} = \left( \begin{array} {cc} I &{\bf X} \\
\overline {\bf X}& I
\end{array} \right).
$$
Set ${\bf Z} = (I -{\bf X} \overline {\bf X})^{-1}$ so that $p(1)
= \det({\bf Z})^{-1}$. Then the inverse of ${\bf M}$ is given by
$$
{\bf M}^{-1}= \left( \begin{array}{cc} {\bf Z} &-{\bf Z} {\bf X}\\
-{\bf {\overline X} Z} & {\bf {\overline Z}} \end{array} \right)
$$
with
$$
{\bf Z} = \frac {1} {p (1)} \left( \begin{array}{cc} 1 - \overline b c - \vert f \vert^2& a
\overline b + b \overline f,\\
c \overline a + f \overline c&1 - \overline c b - \vert a \vert^2 \end{array} \right)
$$
and
$$
{\bf ZX} = \frac {1} {p (1)} \left( \begin{array}{cc} a - a \vert f \vert^2 +bc \overline f&b -
\vert b \vert^2 c + a f \overline b\\
c - \vert c \vert^2 b + af \overline{c}& f - \vert a \vert^2 f + cb \overline a \end{array} \right).
$$
Therefore
$$
\begin{array}{ll}
p (1) \, \omega^1 &= ( 1 - \overline b c - \vert f \vert^2) \eta^1 + \\[5pt]
{}&+(a \overline b + b \overline f) \eta^2 +
(- a + a \vert f \vert^2 -
b c \overline f) \overline \eta^1 + (- b + \vert b \vert^2 c - a f \overline b) \overline \eta^2,\\[5pt]
p (1) \, \omega^2 &= (c \overline a + f \overline c) \eta^1 +
(1 - \overline c b - \vert a \vert^2) \eta^2 +\\[5pt]
{}&+(- c + \vert c \vert^2 b - af \overline c) \overline \eta^1 + (- f + \vert a \vert^2 f - cb \overline a)
\overline \eta^2.
\end{array}
$$
By \cite[Theorem 19]{Ug} a complex structure $J_0$ on the Lie algebra $\mathfrak g$
of $G$
is strong KT if and only if the complex structure $J_0$ is equivalent to the one
defined by
\begin{equation} \label{startingcx}
\left \{ \begin{array} {l}
d \omega^j =0\,, \,\,\,\, j = 1, 2\,, \\[3pt]
d \omega^3 = \rho \omega^{12} + \omega^{1 \overline 1} + G \omega^{1 \overline 2} + H \omega^{2 \overline 2}
\end{array}
\right.
\end{equation}
with $\rho = 0, 1$, $G, H \in \C$  such that $\rho + \vert G \vert^2 = 2 \frak{Re}\, (H)$. Therefore,  comparing with the structure equations \eqref{SKT1},  the  expression of the  matrix
${\bf Y}$  in the new parameters $\rho, G, H$ is the matrix
$$
{\bf Y} = \left ( \begin{array}{cc} 0&-H\\ \rho& G
\end{array} \right).
$$
We will examine separately the two cases: $\rho = 1$ ($J_0$ non abelian) and $\rho =0$ ($J_0$ abelian).\smallskip

For $\rho = 1$, the equation \eqref{integrability} reduces to
$$
- (\det {\bf X}) + \left(\tr ({\bf X} \overline {\bf Y}) - 1 \right) u + \tr \left({\bf X}{\adj}({\bf Y})\right) =0
$$
i.e. to
$$
(f a -c b -c H+ b - aG) -( f \overline G -c \overline H +b - 1)u =0\,.
$$
If we denote by
$$
\begin{array}{lll}
\gamma_1 = - 1 + c + \vert f \vert^2 - G \overline f c + \vert c \vert^2 H\,,&
\gamma_2 = b + G a \overline b - \vert b \vert^2 + H - H \vert a \vert^2\,,\\[3pt]
\gamma_3 = G - f + \overline b f + \overline a c H - \overline b c G\,,&
\gamma_4 = - \overline a - \overline b f + \overline a f \overline G- \overline a c \overline H\,,\\[3pt]
\gamma_5 = 1 - \overline b - G \overline f + \overline c H\,,&
\gamma_6 = af - cb - c H - aG\,,
\end{array}
$$
then we have that given the complex structure $J_0$ equivalent to \eqref{startingcx} with the complex numbers $G, H$
satisfying the condition
$$
1 - 2 {\mathfrak {Re}} H+ \vert G \vert^2 =0\,,
$$
the new complex structure $J_{{\bf X}}$ is integrable and strong KT if and only if the following equations hold
\begin{equation} \label{hypersurfaceE=1}
\left \{
\begin{array}{l}
\overline \gamma_5 u + \gamma_6 =0\,,\\[5pt]
(1 + \vert u \vert^2) \left(\vert \gamma_3 \vert^2 + \vert \gamma_4 \vert^2 + 2 {\mathfrak{Re}}
(\overline \gamma_1 \gamma_2) \right) + \vert \gamma_5 \vert^2 \left (1 - \vert u \vert^2)^2 + \right.\\[5pt]
 +4 {\mathfrak{Re}} \left.(u (\overline \gamma_3 \gamma_4 - \overline \gamma_1 \overline \gamma_2) \right) =0\,,
\end{array}
\right.
\end{equation}
since in terms of the $\gamma_j$, $j = 1, \ldots, 5$, we have that
$$
\begin{array}{lcl}
p(1) d \eta^3 &=&\gamma_5 (1 - \vert u \vert^2) \eta^{12} + (-\gamma_1 + u \overline \gamma_1)
\eta^{1 \overline 1}+(\gamma_3 + u \gamma_4 ) \eta^{1 \overline 2}\\[3pt]
&& +(\overline \gamma_4 + u \overline \gamma_3) \eta^{ \overline 1 2} +(-\gamma_2 + u \overline \gamma_2)
\eta^{\overline 2 2 }.
 \end{array}
$$
Assuming that $\gamma_5 \neq 0$ and by using the first equation of \eqref{hypersurfaceE=1}, we may eliminate
the complex parameter $u$. Then, the second equation of \eqref{hypersurfaceE=1} becomes
$$
\begin{array}{l}
(\vert \gamma_5 \vert^2 + \vert \gamma_6 \vert^2) \left(\vert \gamma_3 \vert^2 + \!\vert \gamma_4 \vert^2 +
2 {\mathfrak{Re}} (\overline \gamma_1 \gamma_2) \right) + \left (\vert \gamma_5 \vert^2 -
\vert \gamma_6 \vert^2)^2+\right.\\[5pt]
- 4 {\mathfrak{Re}} \left.(\gamma_5 \gamma_6 (\overline \gamma_3 \gamma_4 -
\overline \gamma_1 \overline \gamma_2) \right) =0\,,
\end{array}
$$
which is a real equation in the complex variables $a, b, c$ and $f$ and thus defines a real hypersurface of degree $4$ in $\C^4$, non singular at the point $O=(a =0, b =0, c=0, f =0)$.
Therefore, we have that the complex structure $J_{\bf X}$ (deformation of $J_0$), defined by the $(1,0)$-forms
$$
\left \{ \begin{array}{l}
\eta^1 = \omega^1 + a \overline \omega^1 + b \overline \omega^2\,,\\[5pt]
\eta^2 = \omega^2 + c \overline \omega^1 + f \overline \omega^2\,,\\[5pt]
\eta^3 = \omega^3 - \frac{\gamma_6}{\overline \gamma_5} \overline \omega^3\,,
\end{array}
\right.
$$
 has a compatible strong KT metric if and only if $(a,b,c, f)$ belongs to the previous hypersurface.
This completes the case $\rho = 1$.\smallskip

For $\rho =0$, the equation \eqref{integrability} reduces to
$$
\tr ({\bf X} \overline {\bf Y} ) u + \tr \left({\bf X}\,{\adj}({\bf Y}) \right) =0\,,
$$
or equivalently to
$$
-f u \overline G -c H+c u \overline H +b -b u -a G=0\,.
$$
We set
$$
\begin{array} {lll}
\delta_1 = \vert b \vert^2 - \overline G \overline a b - \overline H + \overline H \vert a \vert^2\,,&
\delta_2 = - 1 + \vert f \vert^2 - \overline G f \overline c + \overline H \vert c \vert^2\,,\\[3pt]
\delta_3 = G + \overline b f - G \overline b c + c \overline a H\,,&
\delta_4 = - \overline b f + f \overline a \overline G - c \overline a \overline H\,,\\[3pt]
\delta_5 = - \overline f G + \overline c H - \overline b\,, &
\delta_6 =- c H + b - a G\,.
\end{array}
$$
Then, we have that, given the complex structure $J_0$ equivalent to \eqref{startingcx}
with $\rho =0$ and $G, H$ such that $\vert G \vert^2 = 2 \frak{Re}\, (H) $, the new almost complex structure
$J_{\bf X}$ is integrable and strong KT if and only if the following equations hold
\begin{equation} \label{hyperE=0}
\left\{ \begin{array}{l}
\overline \delta_5 u + \delta_6 =0\,,\\[5pt]
(1 + \vert u \vert^2) (\vert \delta_3 \vert^ 2 + \vert \delta_4 \vert^2 - 2 \frak{Re}\,
(\delta_1 \overline \delta_2) ) + \vert \delta_5 \vert^2 (1 - \vert u \vert^2)^2 +\\[5pt]
 +4 \frak{Re}\, (u ( \delta_1 \delta_2 + \overline \delta_3 \delta_4 + \delta_4 \overline \delta_3) ) =0
\end{array}
\right.
\end{equation}
since
$$
\begin{array}{lcl}
p(1) d \eta^3 &=&(\delta_1 u - \overline \delta_1) \eta^{2 \overline 2}
+(\delta_2 u - \overline \delta_2) \eta^{1 \overline 1}
+(\delta_4 u + \delta_3) \eta^{1 \overline 2} \\[3pt]
&&+(- \overline \delta_3 u - \overline \delta_4) \eta^{2 \overline 1} +(\overline \delta_6 u + \delta_5)
\eta^{12}.
\end{array}
$$
Assuming that $\delta_5 \neq 0$ and by using the first equation of \eqref{hyperE=0} we may eliminate the
complex parameter $u$. Then, the second equation of \eqref{hyperE=0} becomes the real equation in the complex variables $a,b,c,f$
$$
\begin{array}{l}
(\vert \delta_5 \vert^2 + \vert \delta_6 \vert^2) (\vert \delta_3 \vert^ 2 +
\vert \delta_4 \vert^2 - 2 \frak{Re}\, (\delta_1 \overline \delta_2) ) + (\vert \delta_5 \vert^2 -
\vert \delta_6 \vert^2)^2 +\\[5pt]
-4 \frak{Re}\, ( \delta_5 \delta_6 ( \delta_1 \delta_2 + \
\overline \delta_3 \delta_4 + \delta_4 \overline \delta_3) ) =0
\end{array}
$$
which defines a real hypersurface of degree $4$ in $\C^4$, singular at the point $(a=0, b=0, c=0, f=0)$.
In this way we prove that $J_{\bf X}$, deformation of $J_0$, defined by the $(1,0)$-forms
$$
\left \{ \begin{array}{l}
\eta^1 = \omega^1 + a \overline \omega^1 + b \overline \omega^2,\\[3pt]
\eta^2 = \omega^2 + c \overline \omega^1 + f \overline \omega^2,\\[3pt]
\eta^3 = \omega^3 - \frac{\delta_6}{\overline \delta_5} \overline \omega^3,
\end{array}
\right.
$$
has a compatible strong KT metric if and only if $(a,b,c, f)$ belongs to the previous
hypersurface. This completes the case $\rho =0$. Then the theorem is proved.
\end{proof}

\section{Locally conformally balanced structures}
In general, if a Hermitian manifold $(M, J, g)$ is compact, then by using its fundamental $2$-form $F$, one has two
natural linear operators acting on  differential forms:
$$
L \varphi = F \wedge \varphi,
$$
and the adjoint operator $L^*$ of $L$ with respect to the global scalar product defined by
$$
< \varphi, \psi > =p! \int_M (\varphi, \psi) {\mbox {vol}}_g,
$$
where $(\varphi, \psi)$ is the poinwise $g$-scalar product and ${\mbox {vol}}_g$ is the volume form.

As in the strong KT case, the astheno-K\"ahler condition on a compact Hermitian manifold is complementary to the
balanced one, since by \cite[Theorem 1.1]{MT}, for $n \geq 3$, one has
\begin{equation} \label{leeformcondition}
 {L^*}^{n - 1} (2 i \partial \overline \partial F^{n - 2} )= 4^{n - 1} (n - 1) ! (n - 2) [ 2 (n -2)
d^*  \theta + 2 \vert \vert \theta \vert \vert^2 - \vert \vert T \vert \vert^2 ],
\end{equation}
where $\theta = J d^* F$ is the {\em Lee form}, $d^* \theta$ its
co-differential, $\vert \vert \theta \vert \vert$ its $g$-norm and $T$ is the torsion of the Chern
connection $\nabla^C$ on $(M, g)$. \newline
Therefore, if $(J, g)$ is balanced, then $\theta =0$ and, consequently,
the astheno-K\"ahler
condition implies that $T =0$, i.e. $g$ is K\"ahler.


By \cite{IP,P}, a conformally balanced strong KT structure on a compact manifod of complex dimension $n$ whose Bismut
connection has (restricted) holonomy contained in $SU(n)$ is necessarily K\" ahler. We now prove a similar result for
the astheno-K\"ahler metrics.
\begin{theorem} \label{confbalastheno}
A conformally balanced astheno-K\"ahler structure $(J, g)$ on a compact manifold of
complex dimension $n \geq 3$ whose Bismut connection has (restricted) holonomy contained in $SU(n)$ is
necessarily K\"ahler
and therefore it is a Calabi-Yau structure.
\end{theorem}
\begin{proof} Since the Hermitian structure is astheno-K\"ahler, then by \eqref{leeformcondition}, we have
$$
2 (n - 2) d^* \theta + 2 \vert \vert \theta \vert \vert^2 - \vert \vert T \vert \vert^2 =0\,.
$$
Therefore,
\begin{equation} \label{deltatheta}
d^* \theta = \frac{1} {2 (n - 2)} [ \vert \vert T \vert \vert ^2 - 2 \vert \vert \theta \vert \vert ^2]\,.
\end{equation}
By \cite[formula (2.11)]{AI}, the trace $2u$ of the Ricci form of the Chern connection is related to
the trace $b$ of the Ricci form of the Bismut connection by the equation
\begin{equation} \label{tracericci}
2 u = b + 2 (n -1) d^* \theta + 2 (n - 1)^2 \vert \vert \theta \vert \vert^2\,.
\end{equation}
We recall that the condition that the Bismut connection has (restricted) holonomy contained
in $SU(n)$ implies that the Ricci form of the Bismut connection vanishes and, if in addition $M$
is compact, then the first Chern class $c_1(M)$ vanishes.

Therefore, by using \eqref{deltatheta} and \eqref{tracericci}, we get
\begin{equation} \label{u}
2 u = \frac{(n -1)} {(n - 2)} [ \vert \vert T \vert \vert^2 -
2 \vert \vert \theta \vert \vert ^2] + 2 (n - 1)^2 \vert \vert
\theta \vert \vert^2
\end{equation}
and then, if $(J, g)$ is not K\"ahler, we must have $u >0$.

Since in addition $(J, g)$ is conformally balanced, then it was shown in
\cite{Str,P} that there exists a
nowhere vanishing holomorphic $(n,0)$-form $ \tilde \alpha$.

Let be $f = - \frac 12 \vert \vert \alpha \vert \vert ^2$ and denote by $L^{\C}$ the complex Laplacian defined by
$$
L^{\C} (f) = \Delta f + (df, \theta)\,,
$$
where $\Delta$ is the standard Laplace operator and $( \, , \, )$ is the scalar product on the forms induced by $g$.
Then, since $\tilde \alpha$ is holomorphic, as in \cite{P} (see also \cite[formula (19)]{IP2}) we have
$$
L (f) = 2 u \, \vert  \vert \tilde \alpha  \vert  \vert^2 + \vert  \vert  \nabla^C \tilde \alpha \vert  \vert^2\,,
$$
where $\nabla^C$ is the Chern connection and $u$ is given by \eqref{u}.

By the fact that $u > 0$ and $\tilde \alpha \neq 0$, it follows that $L^{\C} f > 0$. From the maximum principle,
we have that $f$
is constant which implies that $u=0$. \newline
The theorem is proved.
\end{proof}
\begin{rem}
{\rm The previous theorem for $n =3$ was already proved in \cite{IP,P}.}
\end{rem}
\section{$\T^2$-bundles over complex surfaces}
By \cite{Ha3} a complex (non-K\"ahler) surface diffeomorphic to a
$4$-dimensional compact homogeneous manifold $X =\Theta \backslash L$, where $\Theta$ is a uniform
discrete subgroup of $L$,
and which does
not admit any K\" ahler structure is one of the following:\vskip.1truecm\noindent
a) Hopf surface;\vskip.1truecm\noindent
 b) Inoue surface of type ${\mathcal S}^0$;\vskip.1truecm\noindent
c) Inoue surface of type ${\mathcal S}^{\pm}$;\vskip.1truecm\noindent
d) primary Kodaira surface;\vskip.1truecm\noindent
e) secondary Kodaira surface;\vskip.1truecm\noindent
f) properly elliptic surface with first odd Betti number.\smallskip

A $\T^2$-bundle over the Inoue surface of type ${\mathcal S}^0$ was considered in \cite{FT}
in order to construct a $6$-dimensional compact solvmanifold with a non-trivial
generalized K\"ahler structure. A similar construction
can be
done for any of the non-K\"ahler complex homogeneous surfaces, by using the description of $L$ and $\Theta$ in
\cite{Ha3}. Indeed we can prove the following
\begin{theorem}\label{theorem1} On any non-K\"ahler compact homogeneous complex surface $X= \Theta \backslash L$
there exists a non-trivial compact $\T^2$-bundle
$M$ carrying a locally conformally balanced strong KT metric.
\end{theorem}
\begin{proof} For the Inoue surface of type ${\mathcal S}^0$, we already proved the result. For the surfaces
a), c), d) and e) we may consider respectively the $6$-dimensional Lie algebras:
$$
\begin{array}{l}
{\mathfrak g}_1 = (e^{23},
 e^{31},
 e^{12}, 0, \frac {\pi} {2} e^{64},
\frac {\pi} {2} e^{45}),\\[5pt]
{\mathfrak g}_2 = (e^{12}, 0, e^{14}, e^{24},\frac {\pi}{2} e^{26}, - \frac {\pi}{2} e^{25}),\\[5pt]
{\mathfrak g}_3 =(0,0, e^{12},0, \frac{\pi}{2} e^{46}, - \frac {\pi}{2} e^{45}),\\[5pt]
{\mathfrak g}_4 = \left( e^{24}, - e^{14}, e^{12}, 0,
\frac{\pi}{2} e^{46},
 - \frac {\pi}{2} e^{45} \right).
 \end{array}
$$
endowed with the complex structure $J$, defined by the $(1,0)$-forms
$$
\eta^1 = e^1 + i e^4, \quad \eta^2 = e^2 + i e^3, \quad \eta^3 = e^5 + i e^6
$$
and the inner product
$g$ defined by
$g = \sum_{j = 1}^6 e^j \otimes e^j.$ Thus $g$ is $J$-Hermitian and,
denoting by $F$ the fundamental $2$-form
associated with the Hermitian structure $(J, g)$,
by a direct computation we have that
$Jd F $ is closed and that the Lee form is closed.

For the surface f) we may take the Lie algebra
 $$
 {\mathfrak g}_5 = \left( 2 e^{13}, -2 e^{23},-e^{12}, 0, \frac{\pi}{2} e^{46}, - \frac {\pi}{2} e^{45} \right)
 $$
endowed with the complex structure
$$
\begin{array}{lll}
J e_1 = \frac{1}{2} (e_3 + e_4)\,, &\quad J e_2 = \frac{1}{2} (e_3 - e_4)\,, &\quad J e_3 = - (e_1 + e_2)\,, \\[5pt]
J e_4 = -e_1 + e_2\,, & \quad J e_5 = e_6\,,&\quad J e_6 = -e_5\, .
\end{array}
$$
and the inner product
$$
g = e^1 \otimes e^1 + e^2 \otimes e^2 + 2(e^3 \otimes e^3 + e^4 \otimes e^4) + e^5 \otimes e^5 + e^6 \otimes e^6.
$$
The fundamental form
$$
F = - e^{13} - e^{14} - e^{23} + e^{24} + e^{56}
$$
is such that the $3$-form
$
J d F = 2 e^{123} + 2 e^{124}
$ is closed. Moreover, the Lee form is closed.

For every Lie algebra ${\mathfrak g}_i$ we have that the span of $\{e^1, e^2, e^3, e^4\}$ can be viewed as the
dual of the Lie algebra of the Lie group $L$.

Let $\H_3$ be the real $3$-dimensional Heisenberg Lie group and $H = \R \ltimes_{\varphi} \R^2$
be the semidirect product of the groups $\R$ and $\R^2$, where $\varphi : \R \to GL(2, \R)$ is the homomorphism
given by
$$
\varphi (t) = \left ( \begin{array} {cc} \cos (\frac {\pi}{2} t)& \sin(\frac {\pi}{2} t)\\
- \sin (\frac {\pi}{2} t) & \cos (\frac {\pi}{2} t) \end{array} \right)\,.
$$
We have
$$
G_1 = SU(2) \times H\,, \quad G_3 = \H_3 \times H\,, \quad G_5 = \widetilde {SL(2, \R)} \times H
$$
and a uniform discrete subgroup of $H$ is of the form $\Gamma' =\Z \ltimes_{\varphi} \Z^2$. Therefore
the Lie groups $G_i$, $i = 1,3,5$ admit a uniform
discrete subgroup.

Let $(m_{jk}) \in SL (2, \Z)$ with two real positive eigenvalues $a$ and $b$ and
denote by $(a_1, a_2)$ and $(b_1, b_2)$ the corresponding eigenvectors.
The remaining Lie groups $G_2$ and $G_4$ are the semidirect products (see \cite{Ha3} for the description of
the corresponding Lie group $L$)
$$
G_2 = \R \ltimes_{\nu} (\H_3 \times \C)\,, \quad G_4 = \R \ltimes_{\tilde \nu} (\H_3 \times \C)\,,
$$
where
the
automorphisms $\nu(t)$ and $\tilde \nu(t)$ are given respectively by
$$
\begin{array} {l}
\nu(t): (x + i y, u,z) \mapsto ( a^t x + i b^t y, u, e^{i \frac{\pi}{2} t} z)\,,\\[4pt]
\tilde \nu(t): (x + i y, u,z) \mapsto (e^{i \frac{\pi}{2} t} (x + i y), u, e^{i \frac{\pi}{2} t} z)\,,
 \end{array}
$$
by identifying the matrix
$$
\left( \begin{array}{ccc} 1&x&u\\ 0&1&y\\ 0&0&1 \end{array} \right)
$$
in $ \H_3$ with $(x + iy, u) \in \C \times \R$.

The Lie group $G_2$ admits a compact quotient by a uniform discrete subgroup of the
form $\Gamma_2 = \Z \ltimes_{\nu} (\tilde \Gamma_n \times \Z^2)$, where $\Z^2$ is the standard
lattice of $\C$ and $\tilde \Gamma_n$ is the lattice of $\H_3$ generated by the elements
$$g_1 = \left( \begin{array}{ccc} 1&a_1&c_1\\ 0&1&b_1\\ 0&0&1 \end{array} \right), \quad
g_2 = \left( \begin{array}{ccc} 1&a_2&c_2\\ 0&1&b_2\\ 0&0&1 \end{array} \right), \quad
g_3 = \left( \begin{array}{ccc} 1&0&c_3\\ 0&1&0\\ 0&0&1 \end{array} \right), \quad c_i \in \R
$$
such that
\begin{itemize}
\item[i)] $[g_1, g_2] = g_3^n$,\medskip
\item[ii)] $\nu (1) (g_1) = g_1^{m_{11}} g_2^{m_{12}} g_3^k\,, \quad \nu(1) (g_2) = g_1^{m_{21}} g_2^{m_{22}} g_3^l\,,$
with $l, k \in \Z$.
\end{itemize}
Let $\Theta_n$ be the discrete sugroup of $\H_3$ defined by
$$
\Theta_n = \left\{ \left( \begin{array}{ccc} 1&a&\frac {c}{n}\\ 0&1&b\\ 0&0&1 \end{array} \right),
\quad a, b, c \in \Z \right\},
$$
then $\Z \ltimes_{\tilde \nu} (\Gamma_n \times \Z^2)$ is a uniform discrete subgroup of the solvable
 Lie group $G_4$.

By construction any quotient $\Gamma_i \backslash G_i$ is a $\T^2$-bundle over the complex surface $\Theta \backslash L$.
\end{proof}

If $X$ is either a Hopf surface or a primary Kodaira surface or a properly elliptic surface with odd
first Betti number, then the $\T^2$-bundle $M$ is a product of two $3$-dimensional manifolds. Then the
interesting cases are the remaining.

The $\T^2$-bundle $\Gamma_4 \backslash G_4$ over the secondary Kodaira surface satisfies the equation
\begin{equation} \label{specialeq}
0 = 2i \partial \overline \partial F = \frac {\alpha'} {4} \trace\, (R^B \wedge R^B),
\end{equation}
 where we denote by $R^B$ the curvature of the Bismut connection and by $F$ the fundamental $2$-form.
This equation is of interest in the context of superstring theory, since it is a particular case of an equation
in the Strominger's system considered in \cite{Str} for Hermitian manifolds of complex dimension three:
\begin{equation}\label{trace}
d H = 2 i \partial \overline \partial F= \frac{\alpha'}{4} \left[ {\mbox \trace}\, (R \wedge R) -
{\mbox \trace}\, (F_A \wedge F_A) \right],
\end{equation}
where $A$ is a Hermitian-Einstein connection on an auxiliary
semi-stable bundle on $M$, $\nabla$ is a metric connection with skew-symmetric torsion $H$ on $M$,
$F_A$ and $R$ denote
respectively the curvature of the two connections $A$ and $\nabla$.\newline
By \cite{Str,P} the Hermitian manifold has to be conformally
balanced
with a holomorphic $(3,0)$-form.

The first solutions of the complete Strominger's system on non-K\"ahler manifolds were constructed
by J.~X. Fu and S.~-T. Yau (see \cite{FY}).

The locally
conformally balanced strong KT manifold $\Gamma_4 \backslash G_4$ gives a solution in dimension $6$ of the
equation \eqref{trace} with $F_A =0$. Indeed, for the Lie algebra ${\mathfrak g}_4$ we have that $JdF= - e^{123}$.
Thus, the non-zero torsion $2$-forms and connections $1$-forms of the Bismut connection $\nabla^B$ are
$$
\begin{array}{l}
\tau^1 = e^{23}\,,\quad \tau^2 = -e^{13}\,, \quad \tau^3 = e^{12}\,,\\[5pt]
\omega^1_2=- e^3 + e^4 = - \omega^2_1\,, \quad
\omega^5_6= - \frac{\pi}{2} e^4 = - \omega^6_5\,.
\end{array}
$$
Therefore, we get that the only non-zero curvature forms for $\nabla^B$ are given by
$$
\Omega^1_2=-e^{12} = - \Omega^2_1
$$
and consequently
$$
\trace\, (\Omega \wedge \Omega) = \sum_{i,j} \Omega^i_j \wedge \Omega^i_j =0\,.
$$
We will show now that $\Gamma_4 \backslash G_4$ does not admit any non-vanishing
holomorphic $(3,0)$-form. By a straightforward computation,
we have that the non-vanishing curvature forms for the Chern connection $\nabla^C$ are
$$
\tilde \Omega^1_2 = - \tilde \Omega^3_4= \frac 12 e^{12}\,,
\qquad \tilde \Omega^5_6 = - e^{12}\,.
$$
Denote by $\tilde R^i_{jhk}$ the curvature components defined by
$$
\tilde \Omega^i_j = {\sum_{h, k=1}^6} {\tilde R^i_{jhk}} e^h \wedge e^k,
$$
and consider the curvature operator $\tilde R(X,Y)$ of the Chern connection defined by
$$
\tilde R(X, Y) Z = {\sum_{i, j, h, k = 1}^6} \tilde R^i_{jhk} (e^h \wedge e^k) (X, Y)\, e^j(Z)\, e_i.
$$
By \cite[Lemma 2, p. 151]{KN}, if $(\Gamma_4 \backslash G_4, J, g)$ admits a non-zero holomorphic
$(3,0)$-form, then the traces of the
two operators $\tilde R(X, Y)$ and $J \circ \tilde R (X, Y)$ must vanish, but, by a direct computation, we have
that the
$$
\trace\, (J \circ \tilde R) (e_1, e_2)=-\pi\,.
$$
Although in physics the  most preferred connection for the anomaly cancellation condition  is the non-Hermitian connection with skew-symmetric torsion equal to  the opposite of  the torsion of $\nabla^B$, also the case of a Hermitian connection may be  interesting.
Indeed, 
we 
 obtain an example, which can be interpreted as a  \lq \lq locally conformal solution\rq \rq of the Strominger's system,  since  locally there is a holomorphic $(3,0)$-form and conformal change to a balanced metric, plus the anomaly cancellation.

\end{document}